\documentclass{article}
\usepackage{amsmath}
\usepackage{amsthm}
\usepackage{pigpen}
\usepackage{hyperref}
\usepackage{amsfonts}
\usepackage{mathpazo}
\usepackage{mathdesign}
\usepackage{amssymb}

\usepackage[bbgreekl]{mathbbol} 
\usepackage{bbm}
\usepackage{amscd}
\usepackage{amsmath,amscd}

\usepackage{verbatim}
\usepackage{enumerate}
\usepackage{accents}
\usepackage{color}
\usepackage{undertilde}
\usepackage[all,2cell]{xy}
\UseAllTwocells
\input xy
\xyoption{2cell}
\xyoption{all}

\usepackage[nottoc,numbib]{tocbibind}

\makeatletter
\global\let\@enddocumenthook\@empty
\makeatother

\AtEndDocument{\bigskip\footnotesize%

M.~Prasma: \textsc{Radboud Universiteit Nijmegen, Institute for
Mathematics, Astrophysics, and Particle Physics, Heyendaalseweg 135, 6525
AJ Nijmegen, The Netherlands} \par
  \textit{E-mail address:}  \texttt{mtnprsm@gmail.com}
}

\numberwithin{equation}{subsection}

\newcommand{\GG}{\mathbb{G}}
\def\Om{{\Omega}}
\def\N{\mathbb{N}}

\def\B{\mathbb{B}}

\def\Del{{\Delta}}

\newcommand{\HSwarrow}{\kern0.05ex\vcenter{\hbox{\Huge\ensuremath{\Swarrow}}}\kern0.05ex}
\newcommand{\hSwarrow}{\kern0.05ex\vcenter{\hbox{\huge\ensuremath{\Swarrow}}}\kern0.05ex}
\newcommand{\LLSwarrow}{\kern0.05ex\vcenter{\hbox{\LARGE\ensuremath{\Swarrow}}}\kern0.05ex}
\newcommand{\LSwarrow}{\kern0.05ex\vcenter{\hbox{\Large\ensuremath{\Swarrow}}}\kern0.05ex}

\newcommand{\HSearrow}{\kern0.05ex\vcenter{\hbox{\Huge\ensuremath{\Searrow}}}\kern0.05ex}
\newcommand{\hSearrow}{\kern0.05ex\vcenter{\hbox{\huge\ensuremath{\Searrow}}}\kern0.05ex}
\newcommand{\LLSearrow}{\kern0.05ex\vcenter{\hbox{\LARGE\ensuremath{\Searrow}}}\kern0.05ex}
\newcommand{\LSearrow}{\kern0.05ex\vcenter{\hbox{\Large\ensuremath{\Searrow}}}\kern0.05ex}

\newcommand{\HDownarrow}{\kern0.05ex\vcenter{\hbox{\Huge\ensuremath{\Downarrow}}}\kern0.05ex}
\newcommand{\hDownarrow}{\kern0.05ex\vcenter{\hbox{\huge\ensuremath{\Downarrow}}}\kern0.05ex}
\newcommand{\LLDownarrow}{\kern0.05ex\vcenter{\hbox{\LARGE\ensuremath{\Downarrow}}}\kern0.05ex}
\newcommand{\LDownarrow}{\kern0.05ex\vcenter{\hbox{\Large\ensuremath{\Downarrow}}}\kern0.05ex}

\newcommand{\HUparrow}{\kern0.05ex\vcenter{\hbox{\Huge\ensuremath{\Uparrow}}}\kern0.05ex}
\newcommand{\hUparrow}{\kern0.05ex\vcenter{\hbox{\huge\ensuremath{\Uparrow}}}\kern0.05ex}
\newcommand{\LLUparrow}{\kern0.05ex\vcenter{\hbox{\LARGE\ensuremath{\Uparrow}}}\kern0.05ex}
\newcommand{\LUparrow}{\kern0.05ex\vcenter{\hbox{\Large\ensuremath{\Uparrow}}}\kern0.05ex}

\newtheorem{thm}{Theorem}[section]
\newtheorem{exam}[thm]{Example}
\newtheorem{cor}[thm]{Corollary}
\newtheorem{lem}[thm]{Lemma}
\newtheorem{pro}[thm]{Proposition}
\newtheorem{defn}[thm]{Definition}

\newtheorem{obs}[thm]{Observation}

\newtheorem{rem}[thm]{Remark}

\DeclareMathOperator{\G}{\mathcal{G}}
\DeclareMathOperator{\F}{\mathcal{F}}
\DeclareMathOperator{\Sp}{\mathcal{S}}
\DeclareMathOperator{\holim}{holim}

\DeclareMathOperator{\id}{id}
\DeclareMathOperator{\op}{op}

\DeclareMathOperator{\map}{map}

\DeclareMathOperator{\AdjCat}{AdjCat}
\DeclareMathOperator{\sS}{s\mathcal{S}}

\def\Sim{\Delta}

\newcommand{\tgpd}{\kern0.05ex\vcenter{\hbox{\footnotesize\ensuremath{2}}}\kern0.05ex\mathcal{G}pd} 

\def\rar{\rightarrow}

\def\lrar{\longrightarrow}
\def\llar{\longleftarrow}

\def\thrar{\twoheadrightarrow}

\def\bu{\bullet}
\def\bar{\overline}

\def\cl{\mathcal}

\newcommand\ackname{\textbf{Acknowledgements}}
\if@titlepage
  \newenvironment{acknowledgements}{
      \titlepage
      \null\vfil
      \@beginparpenalty\@lowpenalty
      \begin{center}
        \bfseries \ackname\
        \@endparpenalty\@M
      \end{center}}
     {\par\vfil\null\endtitlepage}
\else
  
\fi

\title{Segal Group Actions}
\author{Matan Prasma} 

\begin{document}
\maketitle

\begin{abstract}
We define a model category structure on a slice category of simplicial spaces, called the "Segal group action" structure, whose fibrant-cofibrant objects 
may be viewed as representing spaces $X$ with an action of a fixed Segal group (i.e. a group-like, reduced Segal space). We show that this model structure 
is Quillen equivalent to the projective model structure on $G$-spaces, ${\Sp}^{\B G}$, where $G$ is a simplicial group corresponding to the Segal group. 
One advantage of this model is that if we start with an ordinary group action $X\in \Sp^{\B G}$ and apply a weakly monoidal functor of spaces 
$L:\Sp\lrar \Sp$ (such as localization or completion) on each simplicial degree of its associated Segal group action, we get a new Segal group action of $L G$ on $L X$ which can then be rigidified via the above-mentioned Quillen equivalence. 

\end{abstract}

\tableofcontents

\section{Introduction}

The development and use of homotopy-coherent versions of classical notions is by now widespread in several parts of mathematics. It is often beneficial to augment an "up-to-homotopy" notion with a "rigidification" procedure that compares it back to a classical (often enriched) notion. Such a comparison is useful since one can use results that were proven for classical notions in order to establish properties of their homotopy-coherent counter-parts which are usually harder to manage. Let us demonstrate this by the following example.        
For a simplicial group $G$, the methods of higher category theory enable one to have a flexible model of a group action by simply considering the $\infty$-category of $\infty$-functors $Fun(\N (\mathbb{B} G),\N(\Sp^{o}))$ from the homotopy coherent nerve of $\mathbb{B} G$ to the homotopy coherent nerve of the category of spaces. This $\infty$-category can be thought of as spaces with a "group action" up to coherent homotopy, and moreover, comes with a rigidification functor \cite[Proposition 5.1.1.1]{Lur09} to (ordinary) $G$-spaces $$Fun(\N (\mathbb{B} G),\N(\Sp^o))\lrar \Sp^{\mathbb{B} G}.$$ We can use this rigidification to show that
\begin{exam}
The Moore space construction $M(-,n)$ cannot be lifted to an $\infty$-functor $M(-,n):\N(\cl{A}b)\lrar \N(\Sp^o)$. 
\end{exam}
\begin{proof}
If there was a functor $M(-,n):\cl{A}b\lrar \Sp$, it would induce, for every group $G$, an "equivariant Moore space" functor $$M_G(-,n):\cl{A}b^{\B G}\lrar \Sp^{\B G}$$ but \cite{Car} shows that there are (discrete) groups $G$ (e.g. all non-cyclic groups) for which such a functor cannot exist. Similarly, if there was an $\infty$-functor $$M(-,n):\N(\cl{A}b)\lrar N(\Sp^o),$$ it would induce, for any discrete group $G$, an $\infty$-functor $$Fun\left(\N (\B G),\N(\cl{A}b)\right)\lrar Fun\left(\N(\B G),\N(\Sp^o)\right).$$
But the latter may be rigidified to an ordinary functor $$\cl{A}b^G\lrar \Sp^{\B G}$$ which cannot exist by \cite{Car}.

\end{proof}

The purpose of this work is to provide a point-set model for coherent group actions and to establish a rigidification procedure for them. 
We will provide a model-categorical framework for an existing notion, defined and studied in \cite{Pre} under the name "homotopy action" and which will be referred to here as \textbf{Segal group action}.  
A Segal group action aims to encode a coherent action of a loop space $\Om Y$ (together with its coherent homotopies) on a space $X$. More precisely, such an 'action' is a map of simplicial spaces $\pi:A_\bu\rar B_\bu$ in which $A_0\simeq X$, the codomain $B_\bu$ is a Segal group representing $\Om Y$, i.e $B_\bullet$ is a group-like reduced Segal space with $Y\simeq |B_\bullet|$, and certain `Segal-like' maps $$A_n\lrar A_0\times_{B_0}^h B_n$$ are weak equivalences. 

The definition we give here for a Segal group action is simpler than \cite[Defintion 5.1]{Pre} and our first concern is to show that these two definitions coincide. Then, given a Segal group $B_\bu$, we shall construct a model category structure on the slice category $\sS_{/B_\bu}$ whose fibrant-cofibrant objects are precisely the Segal group actions. Using the diagonal functor $d^*:\sS\rar \Sp$ we can also consider the canonical model structure on $\Sp_{/d^*B_\bu}$, induced by slicing under the Kan-Quillen model structure. For $B_\bu$ as above, we further show that 
$$\xymatrix{d^*:\sS_{/B_\bu}\ar[r]<1ex> & \Sp_{/d^*B_\bu}:d_*\ar[l]<1ex>_{\perp}}$$ constitutes a Quillen equivalence between these two model structures. By composing with a Quillen equivalence induced by a  "rigidification map" and the Quillen equivalence of \cite{DFK}, the above-mentioned equivalence shows that the Segal group action model structure is Quillen equivalent to the \textbf{projective model structure} on $\Sp^{\B G}$ (where $G$ is a simplicial group satisfying $BG\simeq d^*B_\bu$), in which weak equivalences (resp. fibrations) are the maps whose underlying map of spaces is a weak equivalence (resp. fibration).     

One technical advantage of Segal group actions is their invariance under a weakly monoidal endofunctor of spaces, namely, functors $L:\Sp\rar \Sp$ which preserve weak equivalences, contractible objects and finite products up to equivalence. Key examples of such functors are localization by a map, $p$-completion a la Bousfield-Kan, and the derived mapping space $map_{\Sp}^h(C,-)$. Applying a weakly monoidal functor $L:\Sp\rar \Sp$ on each simplicial degree of a Segal group action $A_\bu\lrar B_\bu$ yields a new Segal group action $LA_\bu\rar LB_\bu$, now thought of as a coherent action $LB_1$ on the space $LA_0$. This invariance property can be applied (see \ref{tower}) to obtain a Postnikov tower for a $G$-space $X$, composed out of the $P_nX$, but viewed as $P_nG$-spaces.

\subsection*{\textbf{Related work}}

This work complements the treatment of the notion of "homotopy action" which was developed in \cite{Pre}. In the meantime, two related works came out. 
The first is the work of \cite{NSS}, which develops, in the context of an $\infty$-topos, what they called "principal $\infty$-bundles". The treatment of Segal group actions here shows that they constitute a model-categorical presentation of principal $\infty$-bundles (see Corollary \ref{c:drop d_0}). 
The second related work was published recently as \cite{ELS}. There the authors develop the notion of an "$A_\infty$-action" in an operadic manner and thus provide a way to model an action of an \textbf{$\infty$-monoid} on a space.
Although the work in \textit{loc. cit.} does not give a model-categorical framework, it does provide a rigidification result which resembles the one in this paper.

\section{Preliminaries}\label{s: preliminaries}
\begin{enumerate}[(a)]

\item 
Throughout, a \textbf{space} will always mean a simplicial set. Let $\Sp$ (resp. $\Sp_0$) be the category of simplicial sets (resp. reduced simplicial sets) and $\sS$ the category of \textbf{simplicial spaces}; we shall denote an object of $\sS$ with values $[n]\mapsto X_n$ by $X_\bullet$. We let $c_*:\Sp\rar \sS$ be the functor which sends a simplicial set to a degree-wise discrete simplicial space. On the other hand, a simplicial set $K$ may also be viewed as a constant simplicial space which has $K$ in each degree; we shall denote the latter by $K$ again. 

\item
The category $\sS$ is a simplicial category; for $X,Y\in \sS$ we let $$map_{\sS}(X,Y)\in \Sp$$ denote the \textbf{mapping space}. It has the property that for a simplicial set $K$, and simplicial spaces $X,Y$, $$map_{\sS}(K\times X,Y)\cong map_{\Sp}(K,map_{\sS}(X,Y))$$ where $map_{\Sp}(-,-)$ is the \textbf{mapping space} of $\Sp$. If $\Sim^n\in \Sp$ is the standard $n$-simplex, then, by the Yoneda lemma for bisimplicial sets, $c_*\Sim^n$ gives rise to the $n$-th space functor in that $$map_{\sS}(c_*\Sim^n,X)\cong X_n.$$ 

\item
The category $\sS$ is also cartesian closed; for $X,Y\in \sS$ there is an internal-hom object $Y^X\in \sS$ with the property $$\sS(X\times Y, Z)\cong \sS(X,Z^Y).$$ A routine check shows that for a space $K$ and a simplicial space $X$, the two possible meanings for $X^K$ coincide. 

\item By a \textbf{model category structure} we mean a bicomplete category satisfying Quillen's axioms \cite{Qui} and having \textbf{functorial factorizations}. 

\item
We let $\sS_{\text{\tiny{Reedy}}}$ denote the Reedy model structure on simplicial spaces (see \cite{Ree}). This makes $\sS$ into a simplicial combinatorial model category, in which a map $X\rar Y$ in $\sS$ is a Reedy fibration if for each $n\geq 1$, $$map_{\sS}(c_*\Sim^n,Y)\lrar map_{\sS}(c_*\Sim^n, X)\times_{map_{\sS}(c_*\partial\Sim^n,X)}map_{\sS}(c_*\partial\Sim^n,Y)$$ is a Kan fibration.

\item
It is well-known that the Reedy and the injective model structures on $\sS$ coincide (see~\cite[IV.3, Theorem $3.8$]{GJ}) so that every object of $\sS_{\text{\tiny{Reedy}}}$ is Reedy cofibrant. On the other hand, every Reedy fibrant object in $\sS$ has a Kan complex in each simplicial degree, with face maps being Kan fibrations.
For a simplicial space $B_\bu\in \sS$, the Reedy model structure $\sS_{\text{\tiny{Reedy}}}$ induces a (simplicial, combinatorial) model structure, denoted $(\sS_{/B_\bu})_{\text{\tiny{Reedy}}}$, of which all objects are cofibrant and the fibrant objects are precisely Reedy fibrations $A_\bu\thrar B_\bu$. If furthermore $B_\bu$ was Reedy fibrant, then the domain $A_\bu$ of such a fibrant object is also Reedy fibrant.

\item
Similarly, for a fixed space $B$, the Kan-Quillen model structure $\Sp_{\text{\tiny{KQ}}}$ induces a (simplicial, combinatorial) model structure on the slice category, denoted $(\Sp_{/B})_{\text{\tiny{KQ}}}$, of which all objects are cofibrant and the fibrant objects are precisely Kan fibrations $A\thrar B$. As before, if $B$ was a Kan complex, it follows that for every fibrant object $A\thrar B$, the domain $A$ is a Kan complex.

\item
The diagonal functor $d^*:\sS\rar \Sp$ (induced by $d:\Sim\rar \Sim\times \Sim$) is part of an adjoint triple $d_! \dashv d^*\dashv d_*$ (left adjoints on the left). The functor $d_*:\Sp\rar \sS$ is given by $d_*(A)_\bu=A^{\Sim^\bu}$ (i.e. $d_*(A)_n=A^{\Sim^n}$) and the functor $d_!:\Sp\rar \sS$ is defined by extending the formula $d_!(\Sim^n)=\Sim^{n,n}$ via colimits (here $\Sim^{n,n}$ is the representable presheaf on $([n],[n])$). These adjunctions are compatible with the simplicial enrichments on $\sS$ and $\Sp$ mentioned above.   

\item \label{rigid}
There is an adjunction (see \cite{Kan})
\begin{equation}\label{loop group}
\xymatrix{\GG:\Sp_0\ar[r]<1ex> & sGp:B\ar[l]<1ex>_{\perp}}
\end{equation}
where $B$ is the \textbf{classifying space} functor (often denoted by $\bar{W}$) and $\GG$ is the Kan loop group. Furthermore, since the pair \ref{loop group} is in fact a Quillen equivalence, all objects in $\Sp_0$ are cofibrant and all objects in $sGp$ are fibrant, it follows that the unit map of this adjunction $K\lrar  B\GG K$ is a weak equivalence. The category $\Sp_0$ is a reflective subcategory of $\Sp$, with the left adjoint of the pair $$\xymatrix{\widehat{(-)}:\Sp\ar[r]<1ex> & \Sp_0:\iota\ar[l]<1ex>_(0.45){\perp}}$$ defined by identifying all the $0$-simplicies. 
For a connected space $K$, the unit map $K\rar \widehat{K}$ is a weak equivalence, and we shall refer to the composite of these equivalences $\rho: K\rar \widehat{K}\rar B\GG \widehat{K}$ as the \textbf{rigidification} map. The counterpart of the rigidification map relates the loop functor $\Omega:=map_*(S^1,-):\Sp_0\lrar \Sp_0$ to the Kan loop group. 

\item\label{i: rigidification}
For every Kan complex $K\in \Sp_0$ one has a weak equivalence $\Om K\overset{\sim}{\rar} \GG K$ of simplicial sets. Thus, we define an \textbf{$\infty$-group} to be a triple $(\G,B\G,\eta)$ where $\G$ is a space, $B\G$ is a pointed connected space and $\eta:\G \overset{\simeq}{\lrar} \Omega B\G$ is a weak equivalence. We will often refer to $\G$ itself as an $\infty$-group when $B\G$ and $\eta$ are clear from the context.

We say that the composite $$\G\overset{\simeq}{\lrar} \Om B\G\overset{\simeq}{\lrar} \Omega\widehat{B\G} \overset{\simeq}{\lrar} \GG \widehat{B\G}$$ \textbf{rigidifies} $\G$ into a simplicial group.

\item 
A \textbf{Segal space} is a Reedy fibrant\footnote[2]{Notice the slight deviation from the original definition in \cite{Seg}} simplicial space $B_\bullet$ such that for each $n\geq 1$ the Segal maps
$$B_n\lrar \holim (B_1\overset{d_0}{\lrar} B_0\overset{d_1}{\llar}  \cdots \overset{d_0}{\lrar} B_0\overset{d_1}{\llar} B_1) \simeq B_1\times_{B_0}B_1\times_{B_0}\cdots \times_{B_0} B_1\;\; (n\;times),$$ (induced by the maps $p_i:[1]\lrar [n]\;\; 0\mapsto i-1,\;\;1\mapsto i\;\;(1\leq i\leq n)$) 

are weak equivalences.

A Segal space $B_\bullet$ is called a \textbf{Segal groupoid} (or: group-like) if the map $$(d_1,d_0):B_2\lrar \holim(B_1\overset{d_0}{\lrar} B_0 \overset{d_0}{\llar} B_1)\simeq B_1\times_{B_0} B_1$$ is a weak equivalence. If furthermore $B_0\simeq *$ we shall say that $B_\bullet$ is a \textbf{Segal group}.

\item\label{i: Segal}
G.~Segal essentially showed~\cite{Seg} that one can present any $\infty$-group $\G$ as a Segal group.
More precisely, he showed that if $B_\bullet$ is a Segal group, the canonical map $B_1\lrar \Om (d^*B_\bullet)$ is a weak equivalence.  
Given an $\infty$-group $\G$, a \textbf{Segal group for $\G$} is a Segal group $B_\bullet$ together with an equivalence $\G\overset{\sim}{\lrar} B_1$. 

\item
A \textbf{homotopy fiber sequence} is a sequence of spaces $X\lrar Y\lrar Z$ having a null-homotopic composite and such that the associated map to the homotopy fiber $X\rar F_h(Y\rar Z)$ is a weak equivalence.

\item
For a simplicial group $G$, we denote by $\B G$ the simplicial groupoid with one object associated to $G$. We can then consider the category of simplicial functors $\Sp^{\B G}$ and we shall refer to an object $X\in \Sp^{\B G}$ as a \textbf{$G$-space}. We shall refer to the projective model structure on the category of $G$-spaces as the \textbf{Borel model structure} and denote it by $\left(\Sp^{\B G}\right)_{\text{\tiny{Borel}}}$. In other words, this model structure has as weak equivalences (resp. fibrations) the $G$-maps $X\rar X'$ which are weak equivalences (resp. fibrations) in $\Sp_{\text{\tiny{KQ}}}$. The cofibrant objects of $(\Sp^{\B G})_{\text{\tiny{Borel}}}$ are precisely the spaces with a free $G$-action. Thus, given $X\in \Sp^{\B G}$, a model for its cofibrant replacement is $X\times EG$ (where $EG:=WG$ is the free contractible $G$-space) and the \textbf{homotopy quotient} $$X//G:=X\times_{G} EG$$ may be viewed as the right derived functor of the quotient $$(-)/G:(\Sp^{\B G})_{\text{\tiny{Borel}}}\lrar (\Sp_{/BG})_{\text{\tiny{KQ}}}.$$  
Every $G$-space $X$ gives rise to the \textbf{Borel} (homotopy) fiber sequence $$X\lrar X//G\lrar BG$$ and conversely, any (homotopy) fiber sequence of the form $$X\lrar A\lrar BG$$ is equivalent to some Borel fibration. More concisely: 
\begin{thm}\cite{DFK}\label{Borel}
There is a Quillen equivalence $$\xymatrix{(-)\times_{BG}{*}:(\Sp_{/BG})_{\text{\tiny{KQ}}}\ar[r]<1ex> & (\Sp^{\B G})_{\text{\tiny{Borel}}}:(-)/G\ar[l]<1ex>_(0.45){\perp}.}$$ 
\end{thm}

\end{enumerate}

\section{Segal group actions}

Of course, in order to get hands-on calculations, it is useful to have a presentation of the $\infty$-category $Fun\left(\N(\B G),\N(\Sp^o)\right)$ as a model category. One such model is the Borel model structure on $\Sp^{\B G}$ and another is the slice model structure $\Sp_{/BG}$ (see \ref{Borel}). The advantage of the first is that it gives a direct access to the group and the space on which it acts but its disadvantage is that one cannot work with a "flexible" model of the group, e.g., $\Om BG$ nor of the space. In the second model the roles switch in that one may take any space of the homotopy type of $BG$ but there is no direct access to the group $G$ nor to the space on which it acts (which can only be obtained after taking homotopy fiber). As we shall see below, Segal group actions, have, to certain extent, the advantages of both of the models above, since on the one hand a Segal group is a "flexible" model for a simplicial group, and on the other hand, Segal group actions have the homotopy types of the group $G$ and the space on which it (coherently) acts as part of their initial data. We will make use of this advantage to obtain an invariance property of Segal group actions under weak monoidal functors (see \ref{invariance}).

We now come to the main notion of this work. Let $\alpha_0,\alpha_n:[0]\lrar [n]$ be the maps defined by $0\mapsto 0$ and $0\mapsto n$ respectively. Alternatively, $\alpha_0=d^nd^{n-1}\cdots d^1$ and $\alpha_n=d^0\cdots d^0$.
\begin{defn}\label{rel}
A \textbf{Segal group action} is a Reedy fibration of simplicial spaces $ \pi:A_\bullet\lrar B_\bullet $\\
such that:
\begin{enumerate}
\item $B_\bullet$ is a Segal group; 
\item for every $n$, the map $\xymatrix{A_n\ar@{>}[rr]^(0.4){(\alpha_0^*,\pi_n)} 
 && A_0\times_{B_0} B_n}$ is a weak equivalence.
\end{enumerate}
In this case, we say that the Segal group $B_\bu$ \textbf{acts} on $A_\bu$, or that the $\infty$-group $(B_1,|B_\bullet|,\eta:B_1\overset{\simeq}{\lrar}\Om |B_\bullet|)$ 
acts on $A_0$.
\end{defn}

\begin{rem}
One technical advantage of the Reedy fibrancy condition of \ref{rel} is that the structure maps are fibrations. This means that the ordinary notions of fibers, sections etcetara for these maps are homotopy invariant.
\end{rem}

The origin of Definition~\ref{rel} is \cite[Definition 5.1]{Pre} where it was called \textbf{homotopy action}. However, the reader may wonder about a difference between Definition~\ref{rel} and \cite[Definition 5.1]{Pre}. Namely, the definition we give here does not include the condition that the map  $$\xymatrix{A_n\ar@{>}[rr]^(0.4){(\alpha_n^*,\pi_n)}
 && A_0\times_{B_0} B_n}$$ is a weak equivalence. We will now show that this additional condition is implied by the conditions of Definition~\ref{rel} and is thus redundant. This was kindly pointed-out to us by Thomas Nikolaus. The proof we give here is independent.
 
\begin{pro}\label{p:drop d_0}
Let $\pi:A_\bullet\lrar B_\bullet$ be a Segal group action. Then $A_\bullet$ is a Segal groupoid and the map $$\xymatrix{A_n\ar@{>}[rr]^(0.4){(\alpha_n^*,\pi_n)}
 && A_0\times_{B_0} B_n}$$ is a weak equivalence.
\end{pro}

\begin{proof}
Note that our fibrancy assumption implies that the map $$\xymatrix{A_n\ar@{>}[rr]^(0.4){(\alpha_n^*,\pi_n)}
 && A_0\times_{B_0} B_n}$$ is a weak equivalence if and only if the square $$\xymatrix{A_n\ar[r]\ar[d]_{\alpha_n^*} & B_n\ar[d]\\ A_0\ar[r] & B_0}$$ is homotopy cartesian. 

Consider the following commutative cube 

\begin{equation}\label{e:group-like} 
\xymatrix {
    A_2 \ar[rr]^{d_0} \ar[dd]_{d_1} \ar[dr]|-(0.5){\pi_2} && A_1 \ar[dr]|-(0.5){\pi_1} \ar'[d][dd] \\
    & B_2 \ar[rr]|-(0.35){d_0}\ar[dd]|-(0.35){d_1} && B_1 \ar[dd]^{d_0} \\
    A_1 \ar'[r][rr] \ar[dr]|-(0.5){\pi_1} && A_0 \ar[rd]|-(0.5){\pi_0} \\
    & B_1 \ar[rr]_{d_0} && B_0. \\
  }
\end{equation}  
Since $B_\bullet$ is a Segal-group and in particular group-like, the outer face is homotopy cartesian. Consider 
$$\xymatrix{A_2\ar[r]^{\pi_2}\ar[d]_{d_1} & B_2\ar[d]^{d_1} \\ A_1\ar[r]^{\pi_1}\ar[d]_{d_1} & B_1\ar[d]^{d_1}\\ A_0\ar[r] & B_0.}$$ Since $\pi:A_\bullet \lrar B_\bullet$  
is a Segal group action, the lower square is homotopy cartesian, and since $d_1d_1=d_1d_2$ the outer rectangle is homotopy cartesian. It follows that the upper square is homotopy cartesian; this square is the left-hand face of the cube~\ref{e:group-like}. Consider the following commutative diagram of solid arrows. 

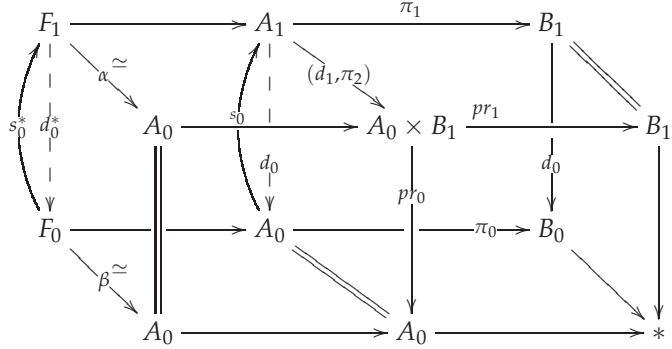
\begin{equation*}\label{e:fibers}
 \xymatrix{
    F_1 \ar[rr] \ar@{-->}[dd]|-{d_0^*} \ar[dr]|-{\alpha}^\simeq   && A_1 \ar[rr]^{\pi_1} \ar@{-->}'[d][dd]|-(0.35){d_0} \ar[dr]|-(0.5){(d_1,\pi_2)} && B_1 \ar@{=}[dr] \ar'[d][dd]|-(0.35){d_0} \\
    & A_0\ar[rr]\ar@{=}[dd] && A_0\times B_1 \ar[rr]^(0.3){pr_1} \ar[dd]|-(0.3){pr_0} && B_1 \ar[dd] \\
    F_0 \ar'[r][rr]\ar[dr]|-{\beta}^{\simeq} \ar@/^1pc/[uu]|-{s_0^*} && A_0 \ar'[r][rr]|-(0.6){\pi_0} \ar@{=}[dr] \ar@/^1pc/[uu]|-(0.5){^{\;}}|-(0.55){^{s_0}} && B_0 \ar[rd] \\
    & A_0\ar[rr] && A_0 \ar[rr] && {*} \\
  }
\end{equation*}

Here, $F_0$ and $F_1$ are fibers of $\pi_0$ and $\pi_1$ (we assume a base-point in $B_1$ was chosen), the maps $s_0^*$ and $d_0^*$ are the ones induces by $s_0$ and $d_0$ and $\beta$ is the map induced between the fibers. Since $B_0$ is contractible, $\beta$ is an equivalence and it follows from $2$-out-of-$3$ that $s_0^*$ is an equivalence. Since $d_0^* s_0^*=\id$ it follows that $d_0^*$ is an equivalence, which means that the lower face of ~\ref{e:group-like} is homotopy cartesian. We now deduce that all the faces of the cube~\ref{e:group-like} are homotopy cartesian and in particular, cartesianess of the inner face means that $A_\bullet$ satisfies the group-like condition. 

Consider now the following commutative cube.

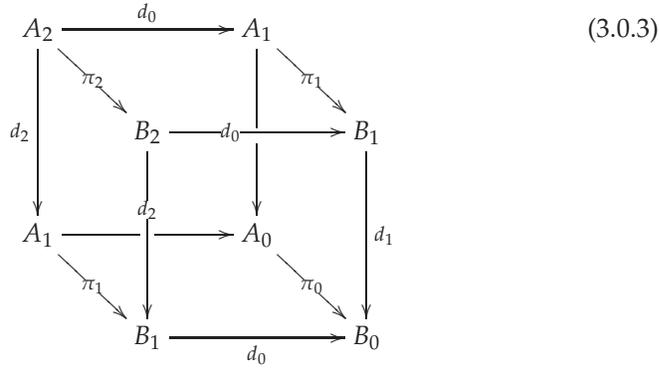
\begin{equation}\label{e:Segal} 
\xymatrix {
    A_2 \ar[rr]^{d_0} \ar[dd]_{d_2} \ar[dr]|-(0.5){\pi_2} && A_1 \ar[dr]|-(0.5){\pi_1} \ar'[d][dd] \\
    & B_2 \ar[rr]|-(0.35){d_0}\ar[dd]|-(0.35){d_2} && B_1 \ar[dd]^{d_1} \\
    A_1 \ar'[r][rr] \ar[dr]|-(0.5){\pi_1} && A_0 \ar[rd]|-(0.5){\pi_0} \\
    & B_1 \ar[rr]_{d_0} && B_0 \\
  }
\end{equation}  

The outer face is homotopy cartesian since $B_\bullet$ is a Segal space and the right-hand face is homotopy cartesian since $\pi:A_\bullet\lrar B_\bullet$ is a Segal group action. We showed that the lower face is homotopy cartesian and it follows that all the faces of~\ref{e:Segal} are homotopy cartesian. In particular, cartesianess of the inner face means that the Segal map for $n=2$ is an equivalence. A similar argument shows that all Segal maps are equivalences (we omit the details for the sake of brevity). It follows that $A_\bullet$ is a Segal groupoid. The homotopy cartesianess of the upper and right-hand faces of the cube \ref{e:group-like} means that the map $$\xymatrix{A_n\ar@{>}[rr]^(0.4){(\alpha_n^*,\pi_n)}
 && A_0\times_{B_0} B_n}$$ is a weak equivalence for $n=1,2$ and a similar argument shows this holds for any $n\geq 1$.

\end{proof}

Put differently, Proposition~\ref{p:drop d_0} shows that Definition~\ref{rel} simplifies \cite[Definition 5.1]{Pre}. In fact, it also allows a comparison between the notion of a Segal group action and an $\infty$-categorical notion of a "group action", as was defined in \cite{NSS}.

\begin{cor}\label{c:drop d_0}
Let $\pi:A_\bullet\lrar B_\bullet$ be a Segal group action, viewed as a functor $\Delta^{\op}\lrar \sS^{[1]}$. Then the underlying $\infty$-functor of $\pi$ is a group action in the sense of \cite[Definition 3.1]{NSS}.
\end{cor}

\begin{exam}\label{bar}
Let $G$ be a simplicial group and $X$ a (right) $G$-space. The Bar construction \cite[\S 7]{May} provides, up to a Reedy fibrant replacement, a Segal group action $Bar_{\bu}(X,G)\lrar Bar_\bu(G)$. The maps 
$$\xymatrix{X\times G^n\ar@{>}[rr]<0.5ex>^(0.4){(\alpha_n^*,\pi_n)} \ar@{>}[rr]<-0.5ex>_(0.4){(\alpha_0^*,\pi_n)}
 & & X\times G^n}$$ are given by the identity and $(x,g_1,...,g_n)\mapsto (xg_1\cdots g_n, g_1,...,g_n)$ (respectively).
\end{exam}

For an $\infty$-group $\G$ together with a fixed choice of a Segal group $B_\bu$ for $\G$, we can thus consider the full subcategory of $\sS_{/B_\bu}$ spanned by the Segal group actions. This category in meant to give a 'soft' model for $G$-actions where $G$ is some simplicial group with $BG\simeq B\G$.  
 
\section{The Segal group action model structure}

Throughout, we fix a Segal group $B_\bu$. 
\begin{defn}\label{Model0}
The \textbf{Segal group action} model structure, $(\sS_{/B_\bu})_{\text{\tiny{SegAc}}}$ is the left Bousfield localization of $(\sS_{/B_\bu})_{\text{\tiny{Reedy}}}$ with respect to the maps 
$$\xymatrix@=0.65pc{c_*\Sim^0\ar[dr]\ar[rr]^{\alpha_0^*} && c_*\Sim^n\ar[dl]^\sigma\\ & B_\bu &}$$

defined for all pairs $(n,\sigma)$ where $n\geq 1$ and $\sigma:c_*\Sim^n\lrar B_\bu$.
\end{defn}
\begin{pro}\label{fc0}
The fibrant-cofibrant objects of $(\sS_{/B_\bu})_{\text{\tiny{SegAc}}}$ are precisely the Segal group actions.
\end{pro}

\begin{proof}
In $(\sS_{/B_\bu})_{\text{\tiny{Reedy}}}$ all objects are cofibrant and since left Bousfield localizations do not change the class of cofibrations, all objects in $(\sS_{/B_\bu})_{\text{\tiny{SegAc}}}$ are cofibrant.

An object $\pi:A_\bu\lrar B_\bu$ is fibrant if and only if it is local with respect to the maps of definition \ref{Model0}. Unwinding the definitions, we see that
$$\map_{/B_\bu}(c_*\Sim^n,A_\bu)\overset{\sim}{\lrar} \map_{/B_\bu}(c_*\Sim^0,A_\bu)$$
$\Leftrightarrow$
\begin{equation}\label{equiv}
Fib(A_n\overset{\pi_n}{\lrar} B_n)\overset{\sim}{\lrar} Fib(A_0\overset{\pi_0}{\lrar} B_0).
\end{equation}

This in turn is the map of associated fibers on vertical arrows in the square
\begin{equation}\label{square}
\xymatrix{A_n\ar[r]\ar[d] & A_0\ar[d]\\B_n\ar[r] & B_0 }
\end{equation}
with the horizontal maps being $\alpha_0^*$. The equivalence \ref{equiv} is precisely the homotopy cartesianess of \ref{square}, which in turn is just the condition that 
$$\xymatrix{A_n \ar@{>}[rr]^{(\alpha_0^*,\pi_n)}
 && A_0\times_{B_0} B_n}$$ are equivalences. Hence, $A_\bu\lrar B_\bu$ is a Segal group action.
\end{proof}

Recall that our goal is to compare the Segal group action model structure to the Borel model model structure. In light of \ref{Borel}, we would like to compare $(\sS_{/B_\bu})_{\text{\tiny{SegAc}}}$ to $(\Sp_{/d^*B_\bu})_{\text{\tiny{KQ}}}$. Before that, it is worth verifying that the latter indeed models the Borel homotopy theory:
\begin{pro}\label{eDFK}
The rigidification map $\rho:d^*B_\bu\lrar BG$ (\textsection 2\ref{rigid}) induces a Quillen equivalence 
$$\xymatrix{\rho_*:(\Sp_{/d^*B_\bu})_{\text{\tiny{KQ}}}\ar[r]<1ex> & (\Sp_{/BG})_{\text{\tiny{KQ}}}:\rho^!\ar[l]<1ex>_(0.45){\perp}}$$
\end{pro} 
\begin{proof} 
The space $d^*B_\bu$ is a $0$-connected Kan complex so that $\rho$ is a weak equivalence between fibrant-cofibrant objects.
\end{proof}
That settled, we recall a standard
\begin{obs}
Let $\cal{C},\cal{D}$ be categories and $$\xymatrix{F:\cal{C}\ar[r]<1ex> & \cal{D}:U\ar[l]<1ex>_{\perp}}$$ an adjoint pair. Then for every object $c\in \cal{C}$ there is an induced adjunction on slice categories $$\xymatrix{F_c:\mathcal{C}_{/ c}\ar[r]<1ex> & \mathcal{D}_{/ F c}:U_{c}\ar[l]<1ex>_{\perp}}$$ where $U_{c}$ is defined by applying $U$ and then pulling back along the unit $1\Rightarrow UF$.
\end{obs}
The above observation is applied directly to our case. By abuse of notation, we will denote the induced adjunction on slice categories as before:
\begin{equation*}\label{adj}
\xymatrix{d^*:\sS_{/B_\bu}\ar[r]<1ex> & \Sp_{/d^*B_\bu}:d_*\ar[l]<1ex>_{\perp}.}
\end{equation*}
We are now at a state to formulate the main assertion of this paper: 

\begin{thm}\label{q.e.}
Let $B_\bu$ be a Segal group. The adjoint pair 
\begin{equation}\label{q.adj}
\xymatrix{d^*:(\sS_{/B_\bu})_{\text{\tiny{SegAc}}}\ar[r]<1ex> & (\Sp_{/d^*B_\bu})_{\text{\tiny{KQ}}}:d_*\ar[l]<1ex>_{\perp}.}
\end{equation}
is a Quillen equivalence.
\end{thm}

We begin with a  
\begin{pro}\label{p: Jardine}
If $B_\bu$ is a Segal group, the space $d^*B_\bu$ is a Kan complex. 
\end{pro}

\begin{proof}
A simplicial space $B_\bu$ satisfying the extension condition with respect to the maps of degree-wise discrete simplicial spaces $c_*\Lambda^n_i\lrar c_*\Del^n$, for $0\leq i\leq n$, is fibrant in the diagonal model structure of \cite[Corollary $1.6$]{Jar}. Moreover, the realization $|B_\bu|$ of such a simplicial space is a Kan complex by \cite[Theorem $2.14$]{Jar}. Since $B_\bu$ is a Segal space, it satisfies the extension condition with respect to $c_*\Lambda^n_i\lrar c_*\Del^n$ for $0<i<n$ and since $B_\bu$ is group-like, it satisfies the extension condition with respect to $c_*\Lambda^n_i\lrar c_*\Del^n$ for $i=n$ and $i=n$. 
\end{proof}

The following is a well-known result, that can be deduced, for example, from \cite[Theorem 5.2]{RSS}. 
\begin{pro}\label{basic adj}
The adjunction $d^*\dashv d_*$ is a Quillen pair $$\xymatrix{d^*:\sS_{\text{\tiny{Reedy}}}\ar[r]<1ex> & \Sp_{\text{\tiny{KQ}}}:d_*\ar[l]<1ex>_(0.45){\perp}.}$$
\end{pro}

\begin{cor}\label{slice adj}
The Quillen pair of proposition \ref{basic adj} induces a Quillen pair on slice model categories
\begin{equation*}
\xymatrix{d^*:(\sS_{/B_\bu})_{\text{\tiny{Reedy}}}\ar[r]<1ex> & (\Sp_{/d^*B_\bu})_{\text{\tiny{KQ}}}:d_*\ar[l]<1ex>_{\perp}.}
\end{equation*}
\end{cor}
We would like to use Corollary \ref{slice adj} as a stepping stone in order to prove that \ref{q.adj} is indeed a Quillen pair. For this, we use the simplicial structure as follows.
\begin{lem}
Let $$\xymatrix{F:\cl{M}\ar[r]<1ex> & \cl{N}:U\ar[l]<1ex>_{\perp}}$$ be an adjoint pair of simplicial model categories in which all objects of $\cl{M}$ are cofibrant. Then $F\dashv U$ is a Quillen pair if and only if $F$ preserves cofibrations and $U$ preserves fibrant objects.
\end{lem} 
Since left Bousfield localization does not change the class of cofibrations, it is clear that $d^*:(\sS_{/B_\bu})_{\text{\tiny{SegAc}}}\lrar (\Sp_{/d^*B_\bu})_{\text{\tiny{KQ}}}$ preserves cofibrations.  
\begin{pro}\label{fib}
For a Segal group $B_\bullet$ the functor
$$d_*:(\Sp_{/d^*B_\bu})_{\text{\tiny{KQ}}}\lrar (\sS_{/B_\bu})_{\text{\tiny{SegAc}}}$$ preserves fibrant objects.
\end{pro}
The proof of Proposition \ref{fib} relies on a folklore result which we address first.
\begin{lem}\label{folklore}
For a simplicial group $G$ and a co-span of $G$-spaces $X\lrar Y\llar Z$, the map $$(X\times_{Y}^h Z)//G\lrar X//G\times_{Y//G}^h Z//G$$ is a weak equivalence.
\end{lem}
\begin{proof}
We have a map of (homotopy) fiber sequences 
$$\xymatrix@=1.2pc{X\times_Y^h Z\ar[r]\ar[d] & (X\times_Y^h Z)//G\ar[r]\ar[d] & BG\ar@{=}[d]\\F_h(p)\ar[r] & X//G\times_{Y//G}^h Z//G\ar[r]^(0.7)p & BG}$$ and it is thus enough to show that $X\times_Y^h Z\lrar F_h(p)$ is a weak equivalence. Consider the $3\times 3$ square
$$\xymatrix@=0.9pc{X//G\ar[r]\ar[d] & BG\ar[d] & {*}\ar[d]\ar[l]\\ Y//G\ar[r] & BG & {*}\ar[l]\\ Z//G\ar[u]\ar[r] & BG\ar[u] & {*}\ar[l]\ar[u]}$$
Taking homotopy limits of all rows and then of the resulting column, gives $X\times_Y^h Z$ and taking homotopy limits of all columns and then of the resulting row, gives $F_h(p)$. The result now follows from commutation of homotopy limits.
\end{proof}

\begin{proof}[of \ref{fib}]
By Ken Brown's lemma, $$d_*:(\Sp_{/d^*B_\bu})_{\text{\tiny{KQ}}}\lrar (\sS_{/B_\bu})_{\text{\tiny{Reedy}}}$$ preserves fibrant objects and it is thus left to verify that for a fibrant object $$A\thrar d^*B_\bu\in (\Sp_{/d^*B_\bu})_{\text{\tiny{KQ}}}$$ the map $d_*(A\thrar d^*B_\bu)$ satisfies condition $(2)$ of definition \ref{rel}. 
Let $P_\bu\in \sS$ be the domain of $d_*(A\thrar d^*B_\bu)$. The $n$-th level of $P_n$ is given by the pullback 
\begin{equation}\label{pb}
\xymatrix@=1.3pc{P_n\ar[r]^p\ar[d]_{\pi_n} & A^{\Sim^n}\ar@{->>}[d] \\ B_n\ar[r] & (d^*B_\bu)^{\Sim^n}.}
\end{equation} 
Since $B_\bullet$ is a Segal group, $d^*B_\bullet$ is a connected Kan complex. Thus, the rigidification map described in \S\ref{s: preliminaries}.\ref{i: rigidification} gives an equivalence $$d^*B_\bu\overset{\simeq}{\lrar} BG$$ for a simplicial group $G$ and we let $X$ be the homotopy fiber $F_h(A\lrar BG)$. By Theorem \ref{Borel}, we have $X//G\simeq A$ so that the square \ref{pb} is equivalent to 
\begin{equation*}
\xymatrix@=1.2pc{X//G\times_{BG}G^n\ar[r]\ar[d] & X//G\ar[d] \\ G^n\ar[r] & BG.}
\end{equation*}
This in turn may by rewritten as 
\begin{equation*}
\xymatrix@=1.2pc{X//G\times_{EG//G}^h G^{n+1}//G\ar[r]\ar[d] & X//G\ar[d] \\ G^n\ar[r] & BG}
\end{equation*}
where $G$ acts on $G^{n+1}$ via the inclusion to the last coordinate $G\lrar G^{n+1}$. 
By Lemma \ref{folklore} we can write $$P_n\simeq X//G\times_{EG//G}G^{n+1}//G\simeq (X\times_{EG}^h G^{n+1})//G\simeq (X\times G^{n+1})//G\simeq X\times G^n$$ (the last equivalence here is a straightforward identification) so that the equivalence is indeed induced by the projection maps $(\pi_n,p)$. Since $P_0\simeq X$ and the face maps $d_i$ of $P_n$ are defined via the above-mentioned pullbacks, it follows that the maps of \ref{rel} (2) are weak equivalences.
\end{proof}

\begin{proof}[of \ref{q.e.}]
We shall show that the unit and counit maps, $1\Rightarrow d_*d^*$ and $d^*d_*\Rightarrow 1$, have weak equivalences as their components when restricted to the categories of fibrant(-cofibrant) objects. Let $A\thrar d^*B_\bu$ be a fibrant object of $\Sp_{/d^*B_\bu}$. As we saw in \ref{fib}, $d_*(A\thrar d^*B_\bu)$ is a Segal group action $P_\bu\lrar B_\bu$ and thus has in each simplicial degree $$F_h(P_n\lrar B_n)\simeq P_0\simeq F_h(A\thrar d^*B_\bu).$$ Thus, by \cite{Pup}, $$F_h(d^*P_\bu\lrar d^*B_\bu)\simeq P_0$$ and it follows by the five lemma that $d^*d_*A\lrar A$ is a weak equivalence over $d^*B_\bu$.
On the other hand, if we are given a Segal group action $A_\bu\lrar B_\bu$, then $$F_h(d^*A_\bu\lrar d^*B_\bu)\simeq A_0$$ and the proof of \ref{fib} shows that we have a weak equivalence of simplicial spaces $$A_\bu\simeq d_*d^*A_\bu,$$ which is compatible with the maps to $B_\bu$ and hence a weak equivalence of Segal group actions.  
\end{proof}

\begin{cor}\label{crux}
There is a Quillen equivalence of simplicial combinatorial model categories
\begin{equation*}\label{e:crux}
\xymatrix{St:(\sS_{/B_\bu})_{\text{\tiny{SegAc}}}\ar[r]<1ex> & (\Sp^{\B G})_{\text{\tiny{Borel}}}:Un\ar[l]<1ex>_{\perp}.}\end{equation*}
\end{cor}
\begin{proof}
One simply compose the Quillen equivalences of Theorem \ref{Borel}, Proposition \ref{eDFK} and Theorem \ref{q.e.} as follows:
$$\xymatrix{ (\sS_{/B_\bu})_{\text{\tiny{SegAc}}}\ar[r]<1ex>^{d^*} & (\Sp_{/d^*B_\bu})_{\text{\tiny{KQ}}}\ar[r]<1ex>^{\rho_*}\ar[l]<1ex>_(0.48){\perp}^{d_*} & (\Sp_{/BG})_{\text{\tiny{KQ}}}\ar[l]<1ex>_(0.45){\perp}^{\rho^!}\ar[r]<1ex>^{(-)\times_{BG}{*}} & (\Sp^{\B G})_{\text{\tiny{Borel}}}\ar[l]<1ex>_(0.48){\perp}^{(-)/G}.}$$
\end{proof}

\section{Invariance properties of Segal group actions}

In algebraic topology one often applies constructions to spaces with a group action.
Of course, for a $G$-space $X$, and an endofunctor of spaces $L:\Sp\lrar \Sp$ there is no canonical group action on $LX$. However, many functors under consideration admit additional properties such as the following.  

\begin{defn}
An endofunctor $L:\Sp\lrar \Sp$ is said to be \textbf{weakly monoidal} if:
\begin{enumerate}
\item $L({*})\sim *$;
\item $L$ preserves weak equivalences; and
\item for any $X,Y\in \Sp$ the map $L(X\times Y)\overset{\sim}{\lrar} LX\times LY$ is a weak equivalence.   
\end{enumerate}
\end{defn}

\begin{exam}
For convenience, we mention a few common weakly monoidal endofunctors of spaces:
\begin{enumerate}
\item
Any (homotopy) (co)localization functor in the sense of \cite{Far}. 
These include the $n$-th Postnikov piece and its dual, sometimes called the $n$-th Whitehead piece.
\item 
The $p$-completion functor $\left(\mathbb{Z}/p\right)_\infty$ \'{a} la Bousfield-Kan.
\item
The (derived) mapping space functor from a fixed space $map_{\Sp}^h(A,-)$.
\end{enumerate}
\end{exam}

Any weak monoidal endofunctor $L:\Sp\lrar \Sp$ takes $\infty$-groups to $\infty$-groups. This is so since for any $\infty$-group $\G$ we can construct a Segal group $B_\bu$ for $\G$ (i.e. with an equivalence $B_1\overset{\sim}{\lrar} \G$ as $\infty$-groups). Applying $L$ on each simplicial degree, we see that $LB_\bu$ becomes a Segal group for $L\G$ so that the latter is again an $\infty$-group. The same argument implies:
\begin{obs}\label{invariance}
If $A_\bu\lrar B_\bu$ is a Segal group action and $L:\Sp\lrar \Sp$ is a weak monoidal endofunctor of spaces, then (the Reedy fibrant replacement of) $LA_\bu\lrar LB_\bu$ is a Segal group action.
\end{obs}

Now let $X\in\Sp^{\B G}$ and denote by  $$A_\bullet(X,G)\lrar B_\bullet(G)$$ the Segal group action $\mathbb{R}Un(X)$ obtained from applying the total right derived functor of~\ref{crux} on $X$. The notation is meant to suggest the equivalent map of simplicial spaces $$Bar_{\bu}(X,G)\lrar Bar_\bu (G).$$ For $L:\Sp\lrar \Sp$ weakly monoidal, denote by $B^L_\bullet(G)$ the Reedy fibrant replacement of $LB_\bullet(G)$ which is a Segal group. Similarly, denote by $$A^L_\bullet(X,G)\lrar B^L_\bullet(G)$$ the Segal group action obtained from replacing $$LA_\bullet(X,G)\lrar B^L_\bullet(G)$$ by a Reedy fibration. This is a fibrant-cofibrant object of $\sS_{/B^L_\bullet(G)}$ so that we can apply $St$ of Corollary~\ref{crux} to obtain a space $\bar{LX}\in\Sp^{\B \bar{LG}}$ where $\bar{LG}$ is the simplicial group obtained from applying the Kan loop group functor $\GG$ on the connected Kan complex $d^*B^L_\bullet(G)$. 
Note that we have a weak equivalence 
\begin{equation}\label{e: rigidification}
LG\overset{\simeq}{\lrar} B^L_1(G)\overset{\simeq}{\lrar}\Om d^*(B^L_\bullet(G))\overset{\simeq}{\lrar} \GG d^*(B^L_\bullet(G))=:\bar{LG}
\end{equation}
where the first map is obtained from the Reedy fibrant replacement, the second is the map of \S~\ref{s: preliminaries}.~\ref{i: Segal} and the third is the "rigidification map" from~\S\ref{s: preliminaries}.~\ref{i: rigidification}. Moreover, the space $\bar{LX}$ is canonically equivalent to $A^L_0(X,G)$ which in turn is equivalent to $LX$. This should be viewed as endowing $LX$ with a coherent action of the $\infty$-group $LG$.    

\begin{exam}\label{e: Postnikov}
Take $L=P_n$, the Postnikov $n$-th piece functor, modeled by \\$cosk_{n+1}(Ex^\infty(-))$. For $X\in \Sp^{\B G}$ we get an action of the simplicial group $\bar{P_nG}$ on $\bar{P_nX}$. 
\end{exam} 

\subsection{Towards an equivariant Postnikov tower for group actions}
A natural question arising from our previous considerations is whether it is possible to extend Example~\ref{e: Postnikov} to obtain an "equivariant Postnikov tower" for any group action.
More specifically, denote $\Gamma_n:=\bar{P_nG}$ so that $\bar{P_nX}$ becomes a $\Gamma_n$-space. When we let $n$ vary, the maps $\Gamma_n\lrar \Gamma_{n-1}$ arising from $P_nG\lrar P_{n-1}G$ are group maps, and we wish to obtain maps $$p_n:\bar{P_nX}\lrar \bar{P_{n-1}X}$$ and $$\tau_n:\bar{X}\lrar \bar{P_{n}X},$$ arising from $P_nX\lrar P_{n-1}X$ and $X\lrar P_nX$ which are $\Gamma_n-\Gamma_{n-1}$-equivariant (here $\bar{X}$ arises from $L=Ex^\infty$), thus giving rise to a tower 
 
\begin{equation}\label{tower}
\xymatrix{ & \overset{\vdots}{\bar{P_{n}X}}\ar[d]^{p_n}\\ & \vdots\ar[d]\\& \bar{P_{2}X}\ar[d]^{p_2}\\& \bar{P_{1}X}\ar[d]^{p_1}
\\\bar{X}\ar[r]^{\tau_0}\ar[ur]^(0.65){\tau_1}\ar[uur]^(0.65){\tau_2}\ar[uuuur]^(0.65){\tau_n} & \bar{P_{0}X}.}
\end{equation}

In order to obtain such a tower, one needs to show that the straightening constructions done in this paper are functorial in an appropriate sense. However the mere existence of the tower~\ref{tower} is not satisfactory, and one would like to know that it converges to the original $G$-space $X$. Moreover, if we instead start with a Segal group action, it is desirable to obtain a similar tower of Segal group actions and to show that these two towers are equivalent in the appropriate sense. 
The difficulty in answering such a question is that this tower is not a diagram in any category of $G$-spaces with a fixed group $G$. Rather, it is a diagram in the \textbf{Grothendieck construction} of the functor $$\Sp^{\B(-)}:sGp\lrar \AdjCat$$ (where $\AdjCat$ stands for categories and adjunctions) which associates to every simplicial group $G$ the category $\Sp^{\B G}$ of $G$-spaces and to a simplicial group map the extension-restriction adjunction. One is then lead to consider the homotopy theory of the\\ Grothendieck construction $$\int_{G\in sGp}\Sp^{\B G}$$ which takes into account the homotopy theory of the base $sGp$, and of each of the fibers $\Sp^{\B G}$. It is convenient to have a model structure that presents the homotopy theory at hand but it is not clear a-priori that such a model structure exists. 

With these questions in mind, the author and Yonatan Harpaz developed in~\cite{HP1} general machinery that, in particular, enables one to obtain a model structure on $\int_{G\in sGp}\Sp^{\B G}$ from the model structure on the base and on each of the fibers. This was further developed in~\cite{HP2}
where the authors showed that an analogous "global" model structure can be constructed for Segal group actions and that the two model structures are Quillen equivalent, thus extending the Quillen equivalence of~\ref{crux} to this case. The main result is then that applying $P_n$ to each simplicial degree as in Example~\ref{e: Postnikov} gives an $n$-truncation functor in the integral model structure for Segal group actions. It follows \cite[\S 5.1]{HP2} that the tower~\ref{tower} of group actions converges to the the initial group action.   

\begin{ackname}
I would like to thank Thomas Nikolaus for a useful discussion (see Proposition~\ref{p:drop d_0}) and to the referee for many valuable remarks. The author was supported by the Dutch Science Foundation (NWO), grant number 62001604.  
\end{ackname}

\end{document}